\pdfoutput=1
\documentclass[10pt]{article}
\usepackage{amssymb}
\usepackage{graphicx}
\usepackage{xcolor} 
\usepackage{tensor}
\usepackage{pgfplots}
\usepackage{fullpage} 
\usepackage[margin=1in]{geometry}
\usepackage{amsmath}
\usepackage{amsthm}
\usepackage{verbatim}
\usepackage{hyperref}
\usepackage{enumitem}
\usepackage{mathrsfs} 
\setlist[enumerate]{leftmargin=1.5em}
\setlist[itemize]{leftmargin=1.5em}

\def\Xint#1{\mathchoice
{\XXint\displaystyle\textstyle{#1}}%
{\XXint\textstyle\scriptstyle{#1}}%
{\XXint\scriptstyle\scriptscriptstyle{#1}}%
{\XXint\scriptscriptstyle\scriptscriptstyle{#1}}%
\!\int}
\def\XXint#1#2#3{{\setbox0=\hbox{$#1{#2#3}{\int}$ }
\vcenter{\hbox{$#2#3$ }}\kern-.6\wd0}}

\def\dashint{\Xint-}

\providecommand{\MR}{\relax\ifhmode\unskip\space\fi MR }

\providecommand{\href}[2]{#2}

\usepackage{seqsplit} 
\definecolor{green}{rgb}{0,0.8,0} 

\newtheorem{theorem}{Theorem}[section]

\newtheorem{lemma}[theorem]{Lemma}
\newtheorem{proposition}[theorem]{Proposition}

\theoremstyle{definition}
\newtheorem{definition}[theorem]{Definition}

\theoremstyle{remark}
\newtheorem{remark}[theorem]{Remark}

\numberwithin{equation}{section}

\newcommand{\nnrm}[1]{{\vert\kern-0.25ex\vert\kern-0.25ex\vert #1 
    \vert\kern-0.25ex\vert\kern-0.25ex\vert}}

\newcommand{\nb}{\nabla}

\newcommand{\alp}{\alpha}

\newcommand{\N}{\mathbb N}

\newcommand{\R}{\mathbb R}

\begin{document}

\title{Sharper bounds on the box-counting dimension of singularities in the hyperdissipative Navier-Stokes system}

\author{Min Jun Jo\thanks{Department of Mathematics, Duke University. E-mail:minjun.jo@duke.edu}}


\renewcommand{\thefootnote}{\fnsymbol{footnote}}

\date{\today}

\renewcommand{\thefootnote}{\arabic{footnote}}


\maketitle

\begin{abstract}
    We study upper bounds on the box-counting dimension of the set of potential singular points in suitable weak solutions to the 3D incompressible hyperdissipative Navier-Stokes system
    \begin{equation*}
    \begin{cases}
    \partial_t u + (-\Delta)^{\alpha}u+(u\cdot \nabla)u+\nabla p = 0, \\
    \operatorname{div} u = 0, 
    \end{cases}
\end{equation*}
for $\alp\in(1,5/4)$. Our main observation is that a classical iteration scheme developed in \cite{KY} and used in \cite{WW} to improve upper bounds for the full Laplacian case can be extended to the hyperdissipative case with properly chosen local quantities that are scale-invariant, despite non-locality of fractional Laplacian. This is achieved by matching up the correct orders of the temporal-spatial scales of the required estimates that effectively quantify $(-\Delta)^{\alp}$ during the iterations. In particular, we adopt the hyperdissipative framework built in the recent breakthrough \cite{CDM} where the upper bounds on the box-counting dimension of the set of potential singularities in $\alp$ are given by
\begin{equation*}
    L(\alp)= \frac{15-2\alp-8\alp^2}{3} \quad \mbox{for}\quad 1<\alp<\frac{5}{4}.
\end{equation*}
In this paper, we generalize the iteration scheme \cite{WW} designed for $\alp=1$ to the case $1<\alp<5/4$, which leads to the newly established bound
\begin{equation*}
    J(\alp)= \frac{36(3-\alpha)(3+2\alpha)(5-4\alpha)}{-64\alpha^3+272\alpha^2-300\alpha+369} \quad \mbox{for} \quad 1<\alp<\frac{5}{4},
\end{equation*}
improving the aforementioned bound $L(\alp)$ obtained in \cite{CDM}. See this picture:

\centering \begin{tikzpicture}
\begin{axis}[ xlabel={$\alpha$},ylabel={ $\operatorname{dim}_B (\mathcal{S})$}
,axis x line=bottom
,axis y line=left
,samples=150, thick
,xmin=1, xmax=13/10, ymin=0, ymax=2 
,legend pos=outer north east
]
\node at (axis cs:1,1.7)[anchor=west] {\small{$L(1)=\frac{5}{3}$}};
\addplot+[mark=none,domain=1:1.3] {5-x*2/3-x^2*8/3};
\addlegendentry{$\mbox{Colombo-De Lellis-Massaccesi \cite{CDM}}$};
\node at (axis cs:1,360/277)[anchor=west] {\small{$J(1)=\frac{360}{277}$}};
\node at (axis cs:1.15,1/4)[anchor=west] {\small{$J(\frac{5}{4})=L(\frac{5}{4})=0$}};
\addplot+[mark=none,domain=1:1.25]
{(-288*x^3+792*x^2+756*x-1620)/(64*x^3-272*x^2+300*x-369)};
\addlegendentry{J.};
\addplot[fill=white,draw=black,mark=*] coordinates{(1,5/3)};
\addplot[fill=white,draw=black,mark=*] coordinates{(1,360/277)};
\addplot[fill=white,draw=black,mark=*] coordinates{(1.25,0)};
\end{axis}
\end{tikzpicture}
\end{abstract}

\setcounter{tocdepth}{1}

\tableofcontents

\section{Introduction}
The incompressible $\alpha$-fractional Navier–Stokes system in $\mathbb{R}^3 \times \mathbb{R}_+$ are given by
\begin{equation}\label{fNSE}
    \begin{cases}
    \vspace{0.1in}
    \partial_t u + (-\Delta)^{\alpha}u+(u\cdot \nabla)u+\nabla p = 0, \\
    \operatorname{div} u = 0,
    \end{cases}
\end{equation}
where $u$ represents the fluid velocity field, $p$ denotes the associated pressure, and $(-\Delta)^{\alpha}$ is the differential operator defined by its Fourier symbol $|\xi|^{2\alpha}$ for $\alpha \geq 0$.

\vspace{0.1in}

\noindent \textbf{$\bullet$ The classical Navier–Stokes equations} 

\vspace{0.1in}

The case $\alpha = 1$ corresponds to the classical Navier–Stokes system, for which it remains unknown whether a unique global smooth solution exists for every smooth and sufficiently decaying (or compactly supported) initial datum $u_0$. In his seminal work \cite{Leray}, Leray developed the theory of global existence of weak solutions to \eqref{fNSE} in the case of full Laplacian $\alpha = 1$. Since then, numerous results have addressed the qualitative properties such solutions might possess. One natural question is whether some of Leray's global weak solutions exhibit regularity up to the level of the initial data. In this context, the notion of suitable weak solutions—those satisfying an additional local energy inequality—was introduced and studied. In particular, alongside the pioneering works \cite{V1,V2} by Scheffer, the so-called partial regularity theory for suitable weak solutions has been intensively developed since the groundbreaking paper \cite{CKN} by Caffarelli–Kohn–Nirenberg, and continues to be an active area of research. We also refer to Lin \cite{F}, who provided a simplified proof of this theory. Given the vast literature on this topic, we do not attempt to present a comprehensive list of references here.

\vspace{0.1in}

\noindent \textbf{$\bullet$ Motivation for the hyperdissipative Navier–Stokes equations with $\alpha > 1$}

\vspace{0.1in}

To gain a different perspective on the regularity problem, one may view fluid evolution as a competition between dissipation and convection, and then ask how strong the dissipation needs to be to ensure global regularity. Assuming that $\alpha$ encodes the intensity of dissipation, such a question is equivalent to asking which values of $\alpha$ guarantee global regularity. A partial answer was provided by Lions \cite{Lions}, who showed that if $\alpha \geq \frac{5}{4}$, the system \eqref{fNSE} admits a global smooth solution for any sufficiently decaying smooth initial data. This motivates the study of how to quantify the regularizing effect of dissipation in terms of $\alpha$, with the goal of bridging the gap between results known for $\alpha = 1$ and those for $\alpha > 5/4$. Throughout this paper, we restrict our attention to the range
\[
1 < \alpha < \frac{5}{4}.
\]

\vspace{0.1in}

\noindent \textbf{$\bullet$ Fractal dimensions of the singular set}

\vspace{0.1in}

Rather than directly proving persistence of regularity or the formation of singularities, one may instead attempt to quantify how potential singularities can be distributed in space-time. To this end, a couple of notions of \textit{fractal dimension} may be adopted. Fractal dimensions are indices that quantify the complexity of a fractal set in terms of how sensitively the “details” of its shape change as the observational scale varies. The term “fractal dimension” was once referred to as “fractional dimension” (see \cite{Besi}, for example). The earliest rigorous notion of fractal dimension is the Hausdorff dimension, introduced by Hausdorff in 1918 \cite{Haus} and later developed by Besicovitch \cite{Besi}. Later, the box-counting dimension—also known as the Minkowski or Minkowski–Bouligand dimension—was rigorously defined by Bouligand in 1928 \cite{Bouligand}, though its heuristic form had already appeared in the work of Minkowski \cite{Minkowski} in 1910. These two notions (see Definitions~\ref{def_haus_dim}–\ref{def_box_dim} for the precise definitions used in this paper, and the subsequent discussion on their differences) have been widely used to describe the structure of the potential singular set of solutions to the incompressible Navier–Stokes equations. See, for example, \cite{CKN,CDM,KP,KY,K,KP2,KO,Ozanski,RS,V1,V2,TY,WW,WY} and references therein. Specifically, one studies upper bounds on the parabolic Hausdorff and parabolic box-counting dimensions of the singular set, defined as the set of points in $\mathbb{R}^3 \times \mathbb{R}$ at which the solution may become discontinuous. The idea of analyzing the space-time fractal dimension of the singular set appears to have been introduced by Scheffer \cite{V1,V2} and significantly developed in the work \cite{CKN} by Caffarelli–Kohn–Nirenberg. Earlier, Leray \cite{Leray} had considered time singular sets.

For the box-counting dimension, which is our main focus in this paper, the earliest upper bound for \emph{suitable weak solutions} with $\alpha = 1$ (see Definition~\ref{def:suitable} for $\alpha > 1$) is $5/3 \approx 1.666$, due to Robinson–Sadowski \cite{RS}. This was later refined to $135/82 \approx 1.646$ by Kukavica \cite{K}, by employing an $\varepsilon$-regularity criterion (cf. Proposition~\ref{prop_epsilon}) that is \emph{not} scale-invariant. Kukavica and Pei \cite{KP2} improved the bound to $45/29 \approx 1.55$ by incorporating pressure gradient estimates into the local energy inequality. Subsequently, Koh–Yang \cite{KY} designed a new iteration scheme to lower the bound to $95/63 \approx 1.508$. Building upon \cite{KY}, Wang–Wu \cite{WW} combined the iteration method with pressure gradient estimates (as in \cite{KP2}) to refine the bound to $360/277 \approx 1.3$. This was further improved to $975/758 \approx 1.286$ by Ren–Wang–Wu \cite{RWW}, who used a new $\varepsilon$-regularity criterion introduced by Guevara–Phuc \cite{GP}, indicating that the key to improvement lies in developing better regularity criteria. This approach led to further refinements: He–Wang–Zhou \cite{HWZ} obtained $2400/1903 \approx 1.261$, and the current best known upper bound, $7/6 \approx 1.167$, was established by Wang–Yang \cite{WY}.

\vspace{0.1in}

\noindent \textbf{$\bullet$ Background and previous results for $\alpha \neq 1$}

\vspace{0.1in}

In the recent work \cite{CDM}, which largely inspired this paper, Colombo, De Lellis, and Massaccesi derived an upper bound
\begin{equation}\label{Lbound}
     L(\alpha):=\frac{15 - 2\alpha - 8\alpha^2}{3}, \quad \text{for } 1 < \alpha < \frac{5}{4},
\end{equation}
for the box-counting dimension of the potential singular set of suitable weak solutions to \eqref{fNSE} in the hyperdissipative regime. This provides a quantitative explanation of how singularities may be suppressed by dissipation as $\alpha$ increases. Notably, the bound \eqref{Lbound} recovers $5/3$ when $\alpha = 1$, thereby generalizing the result of Robinson–Sadowski \cite{RS}, which used the classical $\varepsilon$-regularity criterion of Ladyzhenskaya–Seregin \cite{LS}. This opens the possibility of extending the refinements achieved in the aforementioned works for $\alpha = 1$, which motivates the present study. Specifically, we aim to generalize the result of Wang and Wu \cite{WW} to the full range of hyperdissipation $\alpha \in (1, 5/4)$. This extension is natural in light of the fact that \cite{WW} provides the best known upper bound among results relying on Ladyzhenskaya–Seregin type $\varepsilon$-regularity criteria, and such a criterion has recently been established by Colombo–De Lellis–Massaccesi \cite{CDM} (see Proposition~\ref{prop_epsilon}) for the range $\alpha \in (1, 5/4)$.

We also refer to \cite{Ozanski}, where the fractal dimensions of the singular set are studied in the setting of Leray–Hopf weak solutions for $\alpha \in (1, 5/4)$. In addition, \cite{KO} extends the upper bound \eqref{Lbound} to the hypodissipative case $\alpha \in (3/4, 1)$, while \cite{TY} had previously shown that the $(5 - 4\alpha)$-dimensional Hausdorff measure of the singular set is zero for any suitable weak solution in the same regime.

\vspace{0.1in}

\noindent \textbf{$\bullet$ The goal of this paper}

\vspace{0.1in}

The goal of this paper is to provide a more precise description of the potential singularities of suitable weak solutions to \eqref{fNSE} by sharpening the upper bound $L(\alpha)$ established in \cite{CDM}. Specifically, we show that the box-counting dimension of the set of potential singular points is bounded above by
\[
J(\alpha):=\frac{36(3 - \alpha)(3 + 2\alpha)(5 - 4\alpha)}{-64\alpha^3 + 272\alpha^2 - 300\alpha + 369}
\]
for all $\alpha \in (1, 5/4)$. See the following figure:

\vspace{0.3in}

\begin{tikzpicture}
\begin{axis}[ xlabel={$\alpha$},ylabel={ $\operatorname{dim}_B (\mathcal{S})$}
,axis x line=bottom
,axis y line=left
,samples=150, thick
,xmin=1, xmax=13/10, ymin=0, ymax=2 
,legend pos=outer north east
]
\node at (axis cs:1,1.7)[anchor=west] {\small{$L(1)=\frac{5}{3}$}};
\addplot+[mark=none,domain=1:1.3] {5 - x*2/3 - x^2*8/3};
\addlegendentry{Colombo–De Lellis–Massaccesi \cite{CDM}};
\node at (axis cs:1,360/277)[anchor=west] {\small{$J(1)=\frac{360}{277}$}};
\node at (axis cs:1.15,1/4)[anchor=west] {\small{$J(\frac{5}{4})=L(\frac{5}{4})=0$}};
\addplot+[mark=none,domain=1:1.25]
{(-288*x^3 + 792*x^2 + 756*x - 1620)/(64*x^3 - 272*x^2 + 300*x - 369)};
\addlegendentry{$J(\alpha)$};
\addplot[fill=white,draw=black,mark=*] coordinates{(1,5/3)};
\addplot[fill=white,draw=black,mark=*] coordinates{(1,360/277)};
\addplot[fill=white,draw=black,mark=*] coordinates{(1.25,0)};
\end{axis}
\end{tikzpicture} \\

This bound $J(\alpha)$ successfully generalizes the result of \cite{WW} to the hyperdissipative regime $1 < \alpha < 5/4$, and sharpens the previous upper bound \eqref{Lbound} given in \cite{CDM}.


\subsection{Statement of our main result}\label{sec:main}

In this section, we collect some definitions and state our main result. Following \cite{CDM}, we first define Leray-Hopf weak solutions and then, based on the notion of Leray-Hopf weak solutions, we employ an ``extension" theorem introduced in \cite{Yang} to define suitable weak solutions. After we recall the definition of the box-counting dimension, we finally state our main result.
\subsubsection{Preliminaries}
We recall the definition of Leray-Hopf weak solutions to \eqref{fNSE} from \cite{CDM}.
\begin{definition}
Let $u_0\in L^2(\R^3)$ be a given divergence-free initial condition. A pair $(u,p)$ is a Leray-Hopf weak solution of \eqref{fNSE} on $\R^3\times(0,T)$ if the followings hold:
\begin{enumerate}
    \item $u\in L^{\infty}((0,T),L^2(\R^3))\cap L^2((0,T),H^{\alpha})$;
    \item $u$ solves \eqref{fNSE} in the sense of distributions;
    \item $p$ is the potential-theoretic solution of $-\Delta p = \operatorname{div}\operatorname{div}(u\otimes u);$
    \item $u$ satisfies the global energy inequality
    \begin{equation*}
\frac{1}{2}\int|u|^2(x,t)\,dx+\int_{s}^{t}\int |(-\Delta)^{\alpha/2}u|^2(x,\tau)\,dx\,d\tau \leq \frac{1}{2}\int|u_0|^2(x)\,dx\quad
    \end{equation*}
    for almost every $s$ and every $t>s$.
\end{enumerate}
\end{definition}

\noindent To define suitable weak solutions of \eqref{fNSE}, we should leverage the following extension theorem. This is mainly due to the non-locality of the fractional Laplacian involved. Here, we use $\overline{\nabla}$, $\overline{\Delta}$ for operators defined on $\R^{4}_+$ that are analogous to $\nabla,$ $\Delta$ respectively.
\begin{proposition}[\cite{Yang}]\label{thm:yang}
Let $u\in H^{\alpha}(\R^3)$ with $\alpha\in(1,2)$ and set $b:=3-2\alpha.$ Define the differential operator $\overline{\Delta}_b$ by
\begin{equation*}
    \overline{\Delta}_bu^{\ast}:=\overline{\Delta}u^{\ast}+\frac{b}{y}\partial_y u^{\ast}=\frac{1}{y^b}\overline{\operatorname{div}}(y^b\overline{\nabla}u^{\ast}).
\end{equation*}
Then there exists a unique extension $u^{\ast}$ of $u$ in the weighted space $L_{loc}^{2}(\R_{+}^{4},y^b)$ which satisfies $\overline{\Delta_b}u^{\ast}\in L^2(\R_{+}^{4},y^b)$ and
\begin{equation}\label{eq_ext}
    \overline{\Delta}_b^2 u^{\ast}(x,y)=0 
\end{equation}
and the boundary conditions
\begin{align*}
    u^{\ast}(x,0)&=u(x)\\
    \lim_{y\to 0}y^{1-\alpha}\partial_y u^{\ast}(x,y)&=0.
\end{align*}
Moreover, there is a constant $c_{\alpha}$ depending only on $\alpha$, such that
\begin{enumerate}
    \item the fractional Laplacian $(-\Delta)^{\alpha}u$ is given by the formula \begin{equation*}
        (-\Delta)^{\alpha}u(x)=c_{\alpha}\lim_{y\to 0}y^b\partial_y \overline{\Delta}_b u^{\ast}(x,y);
    \end{equation*}
    \item the following energy identity holds \begin{equation}\label{identity}
        \int_{\R^3}|(-\Delta)^{\alpha/2}u|^2\,dx = \int_{\R^3}|\xi|^{2\alpha}|\widehat{u}(\xi)|^2\,d\xi=c_{\alpha}\int_{\R_{+}^{4}}y^b|\overline{\Delta}_b u^{\ast}|^2\,dx\,dy
    \end{equation}
    \item for every extension $v\in L_{loc}^2 (\R_{+}^{4},y^b)$ of $u$ with $\overline{\Delta}_b v \in L^2 (\R_{+}^{4},y^b)$
    \begin{equation}\label{minimality}
        \int_{\R_{+}^{4}}y^b|\overline{\Delta}_b u^{\ast}|^2\,dx\,dy \leq \int_{\R_{+}^{4}}y^b|\overline{\Delta}_b v|^2\,dx\,dy.
    \end{equation}
\end{enumerate}
\end{proposition}

\noindent The above extension allows us to define suitable weak solutions. In the below definition, we use Einstein's convention for the sum on repeated indices in the following definition.

\begin{definition}\label{def:suitable} Let $1<\alpha<2.$
A Leray-Hopf weak solution $(u,p)$ on $\R^3\times[0,T]$ is a suitable weak solution if the following inequality holds a.e. $t\in[0,T]$ and all nonnegative test functions $\varphi\in C_c^{\infty}(\R_{+}^{4}\times(0,T))$ with $\partial_y\varphi(\cdot,0,\cdot)=0$ in $\R^3\times(0,T):$
\begin{multline*}
    \int_{\R^3}\varphi(x,0,t)|u|^2\,dx+ 2c_{\alpha}
    \int_{0}^{t}\int_{\R_{+}^{4}}y^b|\overline{\Delta}_b u^{\ast}|^2\varphi\,dx\,dy\,ds \\
\leq \int_{0}^{t}\int_{\R^3}\bigg[|u|^2\partial_t\varphi|_{y=0} + (|u|^2+p)u\cdot\nabla\varphi|_{y=0} \bigg]\,dx\,ds \\ - 2c_{\alpha}\int_{0}^{t}\int_{\R_{+}^{4}}y^b \overline{\Delta}_b u_{i}^{\ast}(2\overline{\nabla}\varphi\cdot \overline{\nabla}u_{i}^{\ast} +u_{i}^{\ast}\overline{\Delta}_b \varphi)\,dx\,dy\,ds,
\end{multline*}
where the constant $c_{\alpha}$ depends only on $\alpha$ and comes from the previous extension proposition.
\end{definition}

To define the box-counting dimension and the parabolic Hausdorff dimension, we first denote parabolic cylinders by $$Q_r(z)=Q_r(x,t)=B_r(x)\times(t-r^{2\alpha},t],$$
for any $z\in \R^3\times (0,T).$ We also set $$Q_r:=Q_r(0,0) \,\,\mbox{and}\,\, B_r:=B_r(0).$$
Let $E$ denote any given subset of $\R^3\times \R$. Then we can define the parabolic Hausdorff dimension as follows. 
\begin{definition}\label{def_haus_dim}
 For fixed $\delta>0,$ denote by $\mathscr{C}(E,\delta)$ the collection of all coverings of $\{Q_{r_j}(z_j)\}$ that covers $E$ with $0<r_j\leq \delta.$  The \emph{parabolic Hausdorff measure} of $E$ is defined as
\begin{equation*}
    \mathcal{H}^{\beta}(E)=\lim_{\delta\to0} \inf_{\mathscr{C}(E,\delta)}\sum_{j}r_j^{\beta}
\end{equation*}
and the \emph{parabolic Hausdorff dimension} of $E$ is defined by
\begin{equation*}
    \operatorname{dim}_{\mathcal{H}}(E)=\inf\{\beta:\mathcal{H}^{\beta}(E)=0\}.
\end{equation*}
\end{definition}
\noindent We notice that the above standard Hausdorff dimension sometimes doesn't capture certain important features of sets. For example, let us consider the analogous $1$-dimensional Hausdorff dimension. Then for $A=\{\frac{1}{n}\}_{n\in\N}\cup\{0\}\subset\R$ which has an accumulation point at $0$ we still have that $A$'s Hausdorff dimension is zero. This motivates us to investigate a stronger notion called the box-counting dimension which, in principle, might provide more information on $E$ compared to the usual Hausdorff dimension.
\begin{definition}\label{def_box_dim} The (upper) \emph{parabolic box-counting dimension} of $E$ is defined as
   \begin{equation*}
       \operatorname{dim}_{\mathcal{B}} (E) = \limsup_{r\to 0} \frac{\log \mathcal{N}(A,r)}{-\log r}
   \end{equation*}
where we denote by $\mathcal{N}(A,r)$ the minimum number of parabolic cylinders $Q_r(z)$ required to cover $A.$
\end{definition}
\noindent We see that if we consider again the aforementioned set $A$ and its analogous $1$-dimensional box-counting dimension, then we can conclude that the box-counting dimension of $A$ is $1/2$, not zero unlike its Hausdorff dimension. In general, one can check that
\begin{equation*}
    \operatorname{dim}_{\mathcal{H}}(E) \leq \operatorname{dim}_{\mathcal{B}}(E)
\end{equation*}
 and the inequality could be strict. See \cite{F} for more details.

 We end this subsection by stating the $\varepsilon$-regularity theorem that serves as a key ingredient of the proof of our main result formulated in the next subsection (see Theorem~\ref{thm:box} and the proof of Proposition~\ref{prop_hyper}).

\begin{proposition}[\cite{CDM}]\label{prop_epsilon}
There exists a positive constant $\varepsilon_0$, depending only on $\alpha$, such that the following holds: if $(u,p)$ is a suitable weak solution to \eqref{fNSE} and satisfies
\begin{equation*}
    \frac{1}{r^{6-4\alpha}}\int_{Q_{r}}|u|^3+|p-[p]_{r}|^{3/2}\,dx\,dt + \mathcal{T}(r) < \varepsilon_0,
\end{equation*}
then $u\in C^{0,\kappa}(Q_{r/2};\R^3)$ for some $\kappa>0.$
\end{proposition}

\subsubsection{Main result}

 We call a point $z\in\R^3\times\R$ \emph{regular} if a given Leray-Hopf weak solution $(u,p)$ is continuous in some cylinder $Q_r(z).$ Any point which is not regular is called \emph{singular.} We denote by $\mathcal{S}$ the set of all singular points.  Now, we state our main theorem: 
 \begin{theorem}\label{thm:box} Let $1<\alpha<5/4$. If $(u,p)$ is a suitable weak solution of \eqref{fNSE}, then the box-counting dimension of $\mathcal{S}$ is bounded by $J(\alpha)$ defined by \begin{equation}\label{Jbound}
    J(\alpha)= \frac{36(3-\alpha)(3+2\alpha)(5-4\alpha)}{-64\alpha^3+272\alpha^2-300\alpha+369}
 \end{equation}
 \end{theorem}
 
 \noindent \emph{Remarks on Theorem \ref{thm:box}.}
 \begin{enumerate}
 \item The result immediately follows from Proposition~\ref{prop_hyper} and such implication is standard; see \cite{K,KY,WW} for examples. Still, we provide a proof in Section~\ref{sec_proof} for the reader's convenience.
   \item The proof works only when $L(\alpha)>0$ due to its nature of \textit{improving the existing bounds}.
   \item The currently available best upper bound for the full Laplacian case $\alpha=1$ is 7/6 given in \cite{WY}. We do not attempt here to generalize such a result since the method of \cite{WY} heavily relies on a particular form of $\varepsilon$-regularity theorem which is \emph{not} currently available for the hyperdissipative case $\alp>1$.
 \end{enumerate}
\noindent For any $f:\R^3\times(0,T)\to[0,\infty)$, we denote by $\mathcal{M}f$ the maximal function
\begin{equation*}
    \mathcal{M}f(x,t)=\sup_{r>0}\frac{1}{r^3}\int_{B_r(x)}f(y,t)\,dy.
\end{equation*}
 We also define the extended cylinder $Q^{\ast}(r)$ by $Q^{\ast}(r)=Q(r)\times[0,r)$. Then our key proposition is stated as the following.
 \begin{proposition}\label{prop_hyper} Let $1<\alpha<5/4$ and set $b:=3-2\alpha$.
 There exists a positive constant $\varepsilon$, depending only on $\alpha,$ such that the following holds: for each $\gamma<L(\alpha)-J(\alpha)$, there exists $0<\rho_0<1$ such that if $(u,p)$ is a suitable weak solution which satisfies
 \begin{equation*}
 \begin{split}
      \int_{Q^{\ast}(z,8\rho)}y^b |\overline{\Delta}_b u^{\ast}|^{2}&\,dz^\ast+\int_{Q(z,8\rho)}\Big(\mathcal{M}|u|^2\Big)^{\frac{3+2\alpha}{3}} \\
&+\int_{Q(z,8\rho)}|\nabla u|^2\,dz+|p-[p]_{2\rho}|^{\frac{3+2\alpha}{3}}+|\nabla p|^{\frac{3+2\alpha}{4}}\,dz \leq (8\rho)^{L(\alpha)-\gamma}\varepsilon
\end{split}
 \end{equation*}
for some $0<\rho<\rho_0,$ then $z$ is a regular point.
 \end{proposition}
 \begin{remark}
     The above estimate is \emph{not} scale-invariant, which drives the improvement from $L(\alpha)$ to $J(\alpha).$
 \end{remark}
 \begin{remark}Non-locality of the fractional Laplacian is encoded in the first two integrals.
 \end{remark}

 \subsection{Proof of Theorem \ref{thm:box}}\label{sec_proof}
 \begin{proof}[Proof of Theorem \ref{thm:box}.]
  Our strategy is to employ Proposition~\ref{prop_hyper}. Assume, for a contradiction, that $$\dim_{\mathcal{B}}(\mathcal{S})>J(\alpha).$$ Then we can find a constant $\delta\in (J(\alpha),\operatorname{dim}_B(\mathcal{S}))$.  Thanks to the dimension of $\operatorname{dim}_B(\mathcal{S})$(see Definition~\ref{def_box_dim}), there exists a positive sequence $a_i\to 0$ such that
\begin{equation}\label{sequence}
    \mathcal{N}(\mathcal{S},a_i) > \Big(\frac{1}{a_i}\Big)^{\delta}.
\end{equation}
Denote by $(z_j)_{j=1}^{\mathcal{N}(\mathcal{S},a_i)}\subset \mathcal{S}$ an arbitrary collection of points in $\mathcal{S}$ satisfying
\begin{equation*}
    Q_{a_i}(z_{k}) \cap Q_{a_i}(z_{l}) = \emptyset
\end{equation*}
for any $1\leq k,l\leq \mathcal{N}(\mathcal{S},a_i) $ with $k\neq l$.  We observe that Theorem \ref{prop_hyper} implies
\begin{equation*}
\int_{Q_{a_i}^{\ast}(z_j)}y^b |\overline{\Delta}_b u^{\ast}|^{2}\,dz^{\ast}+\int_{Q_{a_i}(z_j)}|\nabla u|^2 +\Big(\mathcal{M}|u|^2\Big)^{\frac{3+2\alpha}{3}}+|p-[p]_{2\rho}|^{\frac{3+2\alpha}{3}}+|\nabla p|^{\frac{3+2\alpha}{4}} \,dz > a_i^{L(\alpha)-\gamma}\varepsilon
\end{equation*}
for any $z_j$ and any $\gamma<L(\alpha)-J(\alpha)$. Next, summing it up over $j$ gives
\begin{gather*}
    \sum_{j=1}^{\mathcal{N}(\mathcal{S},a_i)}\left( \int_{Q_{a_i}^{\ast}(z_j)}y^b |\overline{\Delta}_b u^{\ast}|^{2}\,dz^{\ast}+\int_{Q_{a_i}(z_j)}|\nabla u|^2+\Big(\mathcal{M}|u|^2\Big)^{\frac{3+2\alpha}{3}}+|p-[p]_{2\rho}|^{\frac{3+2\alpha}{3}}+|\nabla p|^{\frac{3+2\alpha}{4}} \,dz\right) \\
    \geq \mathcal{N}(\mathcal{S},a_i)a_i^{L(\alpha)-\gamma}\varepsilon.
\end{gather*}
Due to the pairwise disjointness of the family $\{Q_{a_i}(z_j)\}_{j=1}^{\mathcal{N}(\mathcal{S},a_i)}$, one can find an absolute constant $K>0$ such that the left-hand side is bounded by $K$: for example, we have $\nabla p \in L^{\frac{3+2\alpha}{4}}(\R^3\times (0,T)),$ see Proposition~\ref{pressure5/4} in Appendix. The finiteness of $\|\nb u\|_{L^2(\R^3\times(0,T))}$ also can be derived from the simple inequality 
$\|\nb u\|_{L^2(\R^3\times(0,T))}^2 \leq \| u\|_{L^2(\R^3\times(0,T))}^2+\|(-\Delta)^{\frac{\alp}{2}}u\|_{L^2(\R^3\times (0,T))}^2$ by decomposing the Fourier side into $|\xi|\leq 1$ and $|\xi|>1$, combined with the definition of Leray-Hopf weak solution.

For the right-hand side, we use \eqref{sequence} to get
\begin{equation*}
    \mathcal{N}(\mathcal{S},a_i)a_i^{L(\alpha)-\gamma}\varepsilon > a_i^{L(\alpha)-\gamma-\delta}\varepsilon
\end{equation*}
which leads to
\begin{equation*}
    K > a_i^{L(\alpha)-\gamma-\delta}\varepsilon.
\end{equation*}
Pick $\gamma=L(\alpha)-\frac{\delta+J(\alpha)}{2}$. Since $\delta$ has been chosen to be bigger than $J(\alpha),$ we immediately see that such $\gamma$ satisfies $\gamma<L(\alpha)-J(\alpha).$ Moreover, we know that $\delta\leq L(\alpha)$ because $L(\alpha)$ is the previously known upper bound for $\operatorname{dim}_{\mathcal{B}}(\mathcal{S}).$ It means that $\gamma\geq 0,$ which guarantees the validity of our choice of $\gamma.$ Therefore, with the chosen $\gamma$, we have
\begin{equation*}
    K > a_i^{\frac{J(\alpha)-\delta}{2}}\varepsilon.
\end{equation*}
Letting $i\to\infty$ gives a contradiction against the above inequality, as desired.
\end{proof}

\subsection{Outline of the paper}
In Section~\ref{sec_prelim}, we gather up some preliminaries. Key iteration lemmas are proved in Section~\ref{sec_iteration}. We cover the proof of Proposition~\ref{prop_hyper} in Section~\ref{sec_proof_main}. Appendix is provided at the end of this paper.

\subsection{Notations} We denote implicit but absolute constants arising in estimates by $C$, which may vary from line to line.

\subsection*{Data Availability \& Conflict of Interest Statements} Data sharing not applicable to this article as no datasets were generated or analysed during the current study.  The authors declare that they have no conflicts of interest.

\subsection*{Acknowledgements} MJJ was partially supported by NSF DMS-2043024. The author thanks Jae-Myoung Kim for his insightful comments.

\section{Basic lemmas}\label{sec_prelim}

Here we gather some known estimates. One is the local-in-time supremum of $L^2$ energy of $u$ measured on a ball, and the other is the local-in-space fractional Sobolev embedding in the fahsion of Yang's extension introduced in the previous section.

\begin{lemma}\label{lem:EI}
If $(u,p)$ is a suitable weak solution of \eqref{fNSE} in $\R^3\times[-(\frac{4}{3}r)^{2\alpha},0],$ then we have
\begin{multline}\label{EI}
    \sup_{-r^{2\alpha}\leq t \leq 0}\frac{1}{r^{5-4\alpha}}\int_{B_r}|u|^2\,dx + \frac{1}{r^{5-4\alpha}} \int_{Q_r^{\ast}}y^b|\overline{\Delta}_b u^{\ast}|^2\,dx\,dy\,dt \\
    \leq \frac{C}{r^{5-2\alpha}}\int_{Q_{\frac{4}{3}r}}|u|^2\,dx\,dt + \frac{C}{r^{6-4\alpha}}\int_{Q_{\frac{4}{3}r}}|u|\Big(|u|^2-g(t)+|p|\Big)\,dx\,dt \\
    + C r^{5\alpha-2}\int_{-(4r/3)^{2\alpha}}^{0}\sup_{R\geq r}R^{-3\alpha}\dashint_{B_R}|u|^2\,dx\,dt
\end{multline}
for any $g\in L^1([0,(\frac{4r}{3})^{2\alpha}]).$
\end{lemma}

\begin{proof}
It is enough to show that
\begin{multline}\label{unscaledEI}
    \sup_{-(3/4)^{2\alpha}\leq t\leq 0}\int_{B_{3/4}}|u|^2 + \int_{Q_{3/4}^{\ast}}y^b |\overline{\Delta}_b u^{\ast}|^2 \,dx\,dt\,dt \\
    \leq C \int_{Q_1}|u| + C\int_{Q_1} |u| \Big(|u|^2-f(t)+|p|\Big) + C\int_{-1}^{0} \sup_{R\geq 1}R^{-3\alpha}\dashint_{B_R}|u|^2\,dx
\end{multline}
because we can recover \eqref{EI} by applying \eqref{unscaledEI} to the scaled solution
$$u_{\frac{4}{3}r}(x,t):=(4r/3)^{2\alpha-1}u\Big((4r/3) x, (4r/3)^{2\alpha}t\Big)$$
supplemented with
$$f(t):=(4r/3)^{4\alpha-2}g((4r/3)^{2\alpha}t\Big).$$
Indeed, \eqref{unscaledEI} is a direct consequence of Lemma 3.2 of \cite{CDM} and Lemma 3.4 in the same reference. Notice that $\Big(\dashint_{B_R}|u|\,dx\Big)^{2} \leq \dashint_{B_R}|u|^2\,dx$ by H\"older's inequality. The proof is complete.\end{proof}

To exploit the fractional dissipation in a natural way, we need certain fractional Sobolev embeddings. The issue is that the spatially-local Sobolev embeddings for the scale-invariant quantities cannot be directly applied due to The non-locaity of $(-\Delta)^{\alpha}$. This can be overcome by the extension-wise embedding lemma. See Lemma 4.4 in \cite{CDM}.  
\begin{lemma}[\cite{CDM}]
Let $(u,p)$ be a suitable weak solution. For $0<r\leq \frac{\rho}{2}$, there holds 
  \begin{equation*}
      \|u\|_{L^{\frac{6}{3-2\alpha}}(B_r)}^2 \lesssim
    \int_{B_\rho^{\ast}}y^b|\overline{\Delta}_b u^\ast|^2\,dx\,dy + \rho^{\alpha+3}\sup_{R\geq \frac{\rho}{4}}R^{-3\alpha}\dashint_{B_R}|u|^2\,dx \quad \forall \alpha\in(1,\frac{5}{4}).
  \end{equation*}  
\end{lemma}

\section{Iteration lemmas}\label{sec_iteration}
We introduce the following scaling-invariant quantities:
\begin{equation}\label{def_scale}
\begin{split}
    \mathcal{A}(r)&:=\sup_{-r^{2\alpha}\leq t \leq 0}\frac{1}{r^{5-4\alpha}}\int_{B(r)}|u|^2\,dx \\
    \mathcal{C}(r)&:=\frac{1}{r^{6-4\alpha}}\int_{Q_r}|u|^3\,dx\,dt\\
    \mathcal{D}(r)&:=\frac{1}{r^{6-4\alpha}}\int_{Q_r}|p-[p]_{r}|^{3/2}\,dx\,dt\\
    \mathcal{E}(r)&:= \frac{1}{r^{5-4\alpha}}\int_{Q_r^{\ast}}y^b |\overline{\Delta}_b u^{\ast}|^{2}\,dx\,dy\,dt \\
   \mathcal{T}(r)&:=r^{5\alpha-2}\int_{-r^{2\alpha}}^{0}\sup_{R\geq \frac{r}{4}}\frac{1}{R^{3\alpha}}\dashint_{B_R}{|u|^2}\,dt.
\end{split}
\end{equation}
The system \eqref{fNSE} enjoys the scaling property: if $(u,p)$ is a solution pair to \eqref{fNSE}, then $u_{\lambda}(x,t):=\lambda^{2\alpha-1} u (\lambda x, \lambda^{2\alpha}t)$ and $p_\lambda(x,t):= \lambda^{4\alpha-2} p( \lambda x, \lambda^{2\alpha} t)$ solve \eqref{fNSE} as well for any $\lambda>0$. Such a scaling property gives rise to all the quantities in \eqref{def_scale}, except for $\mathcal{E}$ which involves the extension $u^\ast$ of $u$. Setting up $u^{\ast}_{\lambda}:=\lambda^{2\alpha-1} u^\ast (\lambda x, \lambda y, \lambda^{2\alpha} t)$, we see that $u^{\ast}_{\lambda}$ satisfies not only the boundary condition $u^{\ast}_{\lambda}|_{y=0}=u_{\lambda}$, but also the weighted Laplacian \eqref{eq_ext} required for extensions. Due to the uniqueness of the extension $(u_\lambda)^{\ast},$  we obtain that $(u_\lambda)^{\ast} = u_\lambda^{\ast}$. From this spatially $4$-dimensional scaling for $u^{\ast}$ and the scaling for the weight $y^{b}$ via $y^{b}=(\lambda y)^{b} \lambda^{-b}$, the scaling-invariant quantity $\mathcal{E}(r)$ can be derived as in \eqref{def_scale}.


\begin{lemma}\label{interpolation}
Let $\alpha\in(1,5/4)$. Then, for $0<r\leq\frac{1}{2}\rho,$ there holds
\begin{equation}\label{eq_intp1}
    \mathcal{C}(r)\lesssim\Big(\frac{\rho}{r}\Big)^{\frac{15}{2}-6\alpha}\mathcal{A}^{1/2}(\rho)\left(\mathcal{E}(\rho)+\mathcal{T}(\rho)\right)+\Big(\frac{r}{\rho}\Big)^{6\alpha-3}\mathcal{A}^{3/2}(\rho).
\end{equation}
\end{lemma}

\begin{proof}
Let $0<r\leq\frac{1}{2}\rho.$ By the triangular inequality we observe that
\begin{equation}\label{triangle}
    \int_{B_r}|u|^3\,dx \lesssim \int_{B_r} |u-[u]_{\rho}|^3\,dx + \int_{B_r}|[u]_{\rho}|^3\,dx.
\end{equation}
The first summand on the right-hand side can be estimated by H\"older's inequality and the Poincar\'e-Sobolev inequality
\begin{equation*}
    \begin{aligned}
   \int_{B_r} |u-[u]_{\rho}|^3\,dx &\lesssim \Big(\int_{B_r}|u-[u]_\rho|^2\,dx\Big)^{1/2}\Big(\int_{B_r}|u-[u]_\rho|^4\,dx\Big)^{1/2} \\
    &\lesssim \|u\|_{L^2(B_\rho)}r^{2\alpha-\frac{3}{2}} \Big(\int_{B_r}|u-[u]_{\rho}|^{\frac{6}{3-2\alpha}}\,dx\Big)^{\frac{3-2\alpha}{3}}\\
    &\lesssim
\|u\|_{L^2(B_\rho)}\cdot r^{2\alpha-\frac{3}{2}} \Big(\int_{B_\rho^\ast}y^b |\overline{\Delta}_b u^{\ast}|^{2}\,dx\,dy +\rho^{\alpha+3} \sup_{R\geq\frac{\rho}{4}}R^{-3\alpha}\dashint_{B_R}|u|^2\,dx\Big).\end{aligned}
\end{equation*}
For the second summand, we notice that $|[u]_\rho|^3 \lesssim \rho^{-9/2}(\int_{B_\rho} |u|^2)^{3/2}$ by H\"older's inequality and so
\begin{equation*}
    \int_{B_r}|[u]_\rho|^3\,dx \lesssim \frac{r^3}{\rho^{9/2}} \Big(\int_{B_\rho}|u|^2\,dx\Big)^{3/2}.
\end{equation*}
Integrating \eqref{triangle} in time over $(-r^{2\alpha},0)$ and multiplying it by $\frac{1}{r^{6-4\alpha}},$ the above estimates finish the proof.
\end{proof}

\begin{lemma}\label{lem:pressure}
For $0<r\leq\frac{1}{2}\rho,$ there holds
\begin{equation*}
    \mathcal{D}(r)\lesssim \Big(\frac{r}{\rho}\Big)^{4\alpha-\frac{3}{2}}\mathcal{D}(\rho)+\Big(\frac{\rho}{r}\Big)^{6-4\alpha}\mathcal{C}(\rho).
\end{equation*}
\end{lemma}

\begin{proof} Let $\phi$ be a non-negative function supported in $B_{\rho}$, which is identically $1$ on $B_{\rho/2}.$ We decompose $p$ into the sum $p=p_1+p_2$ where $p_1$ is defined by
$$p_1(x,t):=\int_{\R^3}\frac{1}{4\pi|x-y|}\Big\{\partial_{i}\partial_{j}\Big[\Big(u_i-(u_i)_\rho\Big)\Big(u_j-(u_j)_\rho\Big)\phi\Big]\Big\}(y,t)\,dy.$$
By Calderon-Zygmund theory, $p_1$ can be estimated as
\begin{equation*}
    \int_{B_\rho}|p_1|^{3/2}\,dx \leq c \int_{B_\rho}|u|^3\,dx.
\end{equation*}
It leads to
\begin{equation*}
    \int_{B_r}|p_1-[p_1]_r|^{3/2}\,dx \lesssim \int_{B_r}|p_1|^{3/2}\,dx \lesssim \int_{B_\rho}|u|^3\,dx.
\end{equation*}
Therefore, we obtain
\begin{equation}\label{eqp1}
\begin{split}
    \frac{1}{r^{6-4\alpha}}\int_{Q_r}|p_1-[p_1]_r|^{3/2}\,dx\,dt &\lesssim \frac{1}{r^{6-4\alpha}}\int_{-r^{2\alpha}}^{0}\int_{B_r}|u|^3\,dx\,dt \\
    &\leq \Big(\frac{\rho}{r}\Big)^{6-4\alpha}\frac{1}{\rho^{6-4\alpha}}\int_{Q_\rho}|u|^3\,dx\,dt
\end{split}
\end{equation}
On the other hand, we have $\Delta p = \Delta p_1$ in $B_{\rho/2}$ and therefore
$$\Delta p_2 = 0 \quad \mbox{in}\,\,B_{\rho/2}.$$
Being harmonic, $p_2$ satisfies
\begin{align*}
    \sup_{y\in B_r}|p_2-[p_2]_r| &\leq cr \sup_{y\in B_{\rho/2}}|\nabla p_2(y,t)| \\
    &\leq cr\frac{1}{\rho^4}\int_{B_\rho}|p_2-[p_2]_\rho|\,dy \\
    &\leq cr\frac{1}{\rho^3}\Big(\int_{B_\rho}|p_2-[p_2]_\rho|^{3/2}\,dy\Big)^{2/3}.
\end{align*}
It follows that
\begin{equation}\label{eqp2}
    \begin{split}
        \frac{1}{r^{6-4\alpha}}\int_{Q_r}|p_2-[p_2]_{r}|^{3/2}\,dx\,dt &\leq \frac{1}{r^{6-4\alpha}}\int_{-r^{2\alpha}}^{0}\int_{B_r}\Big(\sup_{y\in B_r}|p_2-[p_2]_r|\Big)^{3/2}\,dx\,dt \\
    &\lesssim \frac{1}{r^{6-4\alpha}}\cdot r^3 \cdot r^{3/2} \cdot  \frac{1}{\rho^{9/2}} \int_{-r^{2\alpha}}^{0} \int_{B_\rho} |p_2-[p_2]_\rho|^{3/2}\,dy\,dt \\
    &= \frac{r^{4\alpha-\frac{3}{2}}}{\rho^{9/2}}\int_{-r^{2\alpha}}^{0} \int_{B_\rho} |p_2-[p_2]_\rho|^{3/2}\,dy\,dt \\
    &\leq \Big(\frac{r}{\rho}\Big)^{4\alpha-\frac{3}{2}}\frac{1}{\rho^{6-4\alpha}}\int_{Q_\rho}|p_2-[p_2]_\rho|^{3/2}\,dy\,dt.
    \end{split}
\end{equation}
By the simple identity $p-[p]_r=p_1-[p_1]_r + p_2-[p_2]_r,$ combining it with \eqref{eqp1} and \eqref{eqp2}, we conclude that
\begin{equation*}
    D(r)\lesssim \Big(\frac{r}{\rho}\Big)^{4\alpha-\frac{3}{2}}D(\rho)+\Big(\frac{\rho}{r}\Big)^{6-4\alpha}C(\rho)
\end{equation*}
as desired.
\end{proof}

\section{Proof of Proposition \ref{prop_hyper}}\label{sec_proof_main}
\begin{proof}[Proof of Proposition~\ref{prop_hyper}]
Without loss of generality, we assume that $z=(0,0)$. We further assume that for some fixed $0<\rho<\rho_0<1$
\begin{equation}\label{MainAssumption}
\begin{split}
\int_{Q_{8\rho}^{\ast}}y^b |\overline{\Delta}_b u^{\ast}|^{2}&\,dz^{\ast}+\int_{Q_{8\rho}}|\nabla u|^2\,dz\\
&+\int_{Q_{8\rho}}\Big(\mathcal{M}|u|^2\Big)^{\frac{3+2\alpha}{3}}+|p-[p]_{2\rho}|^{\frac{3+2\alpha}{3}}+|\nabla p|^{\frac{3+2\alpha}{4}}\,dz \leq (8\rho)^{L(\alpha)-\gamma}\varepsilon,
\end{split}
\end{equation}
where $L(\alpha)=\frac{15-2\alpha-8\alpha^2}{3}.$ It suffices to find maximal $\gamma=\gamma(\alpha)$ such that $z$ would be a regular point and then check that $J(\alpha)$ equals $L(\alpha)-\gamma(\alpha).$ We notice that the suitable $\rho$ will be selected later as well.

\vspace{0.05in}

\noindent \textbf{Step 1. Setting up the iteration:}
To invoke the $\varepsilon$-regularity criterion in Proposition \ref{prop_epsilon}, we want to find sufficiently small $r_N>0$ such that
\begin{equation}\label{eq:epsilon}
   \mathcal{C}(r_N)+\mathcal{D}(r_N)+\mathcal{T}(r_N)<\varepsilon_0.
\end{equation}
It proves that $z$ is a regular point. To this end, we design an iteration procedure as follows. For $\eta\geq1$ and $\zeta>0$ that would be determined later, define the sequence
\begin{equation*}
    r_k=\rho^{\eta+k\zeta} \quad \mbox{for}\,\, k=0,1,\ldots,N.
\end{equation*}
The common ratio of this geometric sequence is denoted by $\theta=\rho^{\zeta}.$ We notice that this sequence is a strictly decreasing sequence, and so the proof reduces from finding small $r_N$ to finding large $N>0$ such that \eqref{eq:epsilon} is fulfilled. By applying Lemma \ref{lem:pressure} repeatedly, with the aid of the simple fact that $\mathcal{C}(r)\lesssim \theta^{4\alpha-6}\mathcal{C}(\theta^{-1}r)$, we have
 \begin{equation}\label{eq:iteration}
 \begin{split}
     \mathcal{C}(r_N)+\mathcal{D}(r_N) &\lesssim \theta^{4\alpha-6}\mathcal{C}(r_{N-1})+\mathcal{D}(r_N) \\
    &\lesssim \theta^{4\alpha-6}\mathcal{C}(r_{N-1}) + \theta^{4\alpha-\frac{3}{2}} \mathcal{D}(r_{N-1}) \\
    &\ldots \\
    &\lesssim \sum_{i=1}^{N}\theta^{(4\alpha-\frac{3}{2})(i-1)+4\alpha-6}\mathcal{C}(r_{N-i}) + \theta^{(4\alpha-\frac{3}{2})N}\mathcal{D}(r_0) =: I+II.
 \end{split}
 \end{equation}
Therefore, it suffices to estimate $I,$ $II,$ and $\mathcal{T}(r_N).$
\\
\textbf{Step 2. Estimate on the tail term $\mathcal{T}(r_N).$}  We can leverage the maximal function to control the non-locality of $\mathcal{T}(r_N)$. We first observe that
\begin{equation}\label{simple}
    \sup_{R\geq\frac{r}{4}}\dashint_{B_R} |u|^2 \leq C \int_{B_{\frac{r}{4}}}\dashint_{B_{\frac{r}{2}}(x)}|u|^2\,dy\,dx \leq C \int_{B_{\frac{r}{4}}}\mathcal{M}|u|^2\,dx \quad \mbox{for any}\,\,r>0.
\end{equation}
Then we obtain
\begin{equation}\label{tail}
    \begin{aligned}
    \mathcal{T}(r_N)
    &\leq C r_N^{5\alpha-2}\int_{-(r_N)^{2\alpha}}^{0}\sup_{R\geq \frac{r_N}{4}}R^{-3\alpha}\int_{B_{\frac{r_N}{2}}}\mathcal{M}|u|^2 \,dt\\
    &\leq C r_N^{2\alpha-2} \int_{-(r_N)^{2\alpha}}^{0} \bigg(\int_{B_{\frac{r_N}{2}}}\Big(\mathcal{M}|u|^2\Big)^{\frac{3+2\alpha}{3}}\bigg)^{\frac{3}{3+2\alpha}}r_N^{\frac{6\alpha}{3+2\alpha}}\,dt \\
    &\leq Cr_N^{2\alpha-2+\frac{6\alpha}{3+2\alpha}+\frac{4\alpha^2}{3+2\alpha}} \bigg(\int_{Q_{r_N}}\Big(\mathcal{M}|u|^2\Big)^{\frac{3+2\alpha}{3}} \,dx\,dt\bigg)^{\frac{3}{3+2\alpha}}  \\
    &\leq C\rho^{(\eta+N\zeta)(4\alpha-2)+(L-\gamma)\frac{3}{3+2\alpha}}\varepsilon^{\frac{3}{3+2\alpha}}
    \end{aligned}
\end{equation}
by H\"older's inequality and the assumption \eqref{MainAssumption}. Thus we have
\begin{equation}
    \mathcal{T}(r_N)\leq C \rho^{J_1}\varepsilon^{\frac{3}{3+2\alpha}}
\end{equation}
where we set
\begin{equation*}
    J_1:=(\eta+N\zeta)(4\alpha-2)+(L-\gamma)\frac{3}{3+2\alpha}.
\end{equation*}
\noindent\textbf{Step 3. Energy estimates.} 
Our goal is to use Lemma~\ref{lem:pressure} to control the pressure term $II$. To this end, we need the estimates on $\mathcal{A}(\rho)$ and $\mathcal{E}(\rho)$ in advance. Both terms can be estimated by the local energy inequality which implicitly encodes certain non-locality of $(-\Delta)^{\alpha}$ with the extension $u^{\ast}$. For the term $\mathcal{A}(\rho),$ we employ \eqref{EI} to get
\begin{multline}
    \mathcal{A}(\rho)=\sup_{t\in[-\rho^{2\alpha},0]}\frac{1}{\rho^{5-4\alpha}}\int_{B_\rho}|u|^2 \\
    \leq \frac{C}{\rho^{5-2\alpha}}\int_{Q_{2\rho}}|u|^2 + \frac{C}{\rho^{6-4\alpha}}\int_{Q_{2\rho}}|u|\Big(|u|^2-[|u|^2]_{2\rho}\Big)+\frac{C}{\rho^{6-4\alpha}}\int_{Q_{2\rho}}|u||p-[p]_{2\rho}|\\
    + C\rho^{5\alpha-2}\int_{-(2\rho)^{2\alpha}}^{0}\sup_{R\geq \rho}R^{-3\alpha}\dashint_{B_R}|u|^2\,dx\,dt.
\end{multline}
The first term on the right-hand side is estimated as
\begin{equation}\label{EI1}
\begin{split}
    \frac{1}{\rho^{5-2\alpha}}\int_{Q_{2\rho}}|u|^2 &\leq \frac{1}{\rho^{5-2\alpha}}\Big(\int_{Q_{2\rho}}|u|^{\frac{6+4\alpha}{3}}\Big)^{\frac{3}{3+2\alpha}}|Q_{2\rho}|^{\frac{2\alpha}{3+2\alpha}} \\
    &\leq C \rho^{4\alpha-5+(L-\gamma)\frac{3}{3+2\alpha}}\varepsilon^{\frac{3}{3+2\alpha}}
\end{split}
\end{equation}
by H\"older's inequality and the assumption \eqref{MainAssumption}. For the second term, we use H\"older's inequality and the Poincar\'e-Sobolev inequality combined with \eqref{MainAssumption} to yield
\begin{equation}\label{EI2}
\begin{split}
    \frac{1}{\rho^{6-4\alpha}}\int_{Q_{2\rho}}|u|\Big(|u|^2-[|u|^2]_{2\rho}\Big) &\leq C\rho^{6\alpha-\frac{15}{2}}\Big(\int_{Q_{2\rho}}|u\nabla u|^{\frac{3+2\alpha}{3+\alpha}}\Big)^{\frac{(6+2\alpha)(4\alpha-3)}{4\alpha(3+2\alpha)}}\Big(\int_{Q_{2\rho}}|u|^{\frac{6+4\alpha}{3}}\Big)^{\frac{9-3\alpha}{\alpha(6+4\alpha)}} \\
    &\leq \rho^{6\alpha-\frac{15}{2}+L-\gamma}\varepsilon.
    \end{split}
\end{equation}
The term involving the pressure is computed similarly:
\begin{equation}\label{EI3}
\begin{split}
        \frac{1}{\rho^{6-4\alpha}}&\int_{Q_{2\rho}}|u||p-[p]_{2\rho}| \\
        &\leq C\rho^{4\alpha-6+\frac{4\alpha-3}{2}} \Big(\int_{Q_{2\rho}}|u|^{\frac{6+4\alpha}{3}}\Big)^{\frac{3}{6+4\alpha}}\Big(\int_{Q_{2\rho}}|p-[p]_{2\rho}|^{\frac{3+2\alpha}{3}}\Big)^{\frac{15-12\alpha}{6+4\alpha}}\Big(\int_{Q_{2\rho}}|\nabla p|^{\frac{3+2\alpha}{4}}\Big)^{\frac{16\alpha-12}{6+4\alpha}} \\
    &\leq C\rho^{6\alpha-\frac{15}{2}+L-\gamma }\varepsilon.
    \end{split}
\end{equation}
Here we have used H\"older's inequality, the Poincar\'e-Sobolev inequality, and the assumption \eqref{MainAssumption}. To estimate the fourth term, we observe the simple fact that
$$\sup_{R\geq\frac{r}{4}}\dashint_{B_R} |u|^2 \leq C \int_{B_{\frac{r}{4}}}\dashint_{B_{\frac{r}{2}}(x)}|u|^2\,dy\,dx \leq C \int_{B_{\frac{r}{4}}}\mathcal{M}|u|^2\,dx \quad \mbox{for any}\,\,r>0$$
which leads to
\begin{equation}\label{EI4}
    \begin{split}
        \rho^{5\alpha-2}\int_{-(2\rho)^{2\alpha}}^{0}\sup_{R\geq \rho}R^{-3\alpha}\dashint_{B_R}|u|^2\,dx\,dt.
        &\leq C\rho^{2\alpha-2}\int_{Q_{2\rho}}\mathcal{M}|u|^2\\
        &\leq C \rho^{4\alpha-2}\Big(\int_{Q_{2\rho}}\Big(\mathcal{M}|u|^2\Big)^{\frac{3+2\alpha}{3}}\Big)^{\frac{3}{3+2\alpha}} \\
        &\leq \rho^{4\alpha-2+(L-\gamma)\frac{3}{3+2\alpha}}\varepsilon^{\frac{3}{3+2\alpha}}
    \end{split}
\end{equation}
by H\"older's inequality and the assumption \eqref{MainAssumption}. If we further assume that
\begin{equation}\label{gamma2}
    \gamma\leq \frac{4\alpha-3}{4\alpha}L,
\end{equation}
which allows the exponent of $\rho$ in \eqref{EI1} to be smaller than that of \eqref{EI2},
from the previous estimates \eqref{EI1}, \eqref{EI2}, \eqref{EI3}, and \eqref{EI4} all together,
we conclude that
\begin{equation}\label{A}
    \mathcal{A}(\rho) \leq C \rho^{4\alpha-5+(L-\gamma)\frac{3}{3+2\alpha}}\varepsilon^{\frac{3}{3+2\alpha}}.
\end{equation}
For the term $\mathcal{E}(\rho),$ our assumption \eqref{MainAssumption} immediately implies
\begin{equation}\label{E}
    \mathcal{E}(\rho) \leq C \rho^{4\alpha-5+(L-\gamma)}\varepsilon.
\end{equation}

\noindent \textbf{Step 4. Estimate on the pressure term $II.$} We first claim that
\begin{equation}\label{pressure}
    \mathcal{D}(r) \leq C r^{6\alpha-\frac{15}{2}} \Big(\int_{Q_r} |\nabla p|^{\frac{3+2\alpha}{4}}\Big)^{\frac{8\alpha-6}{3+2\alpha}} \Big(\int_{Q_r}|p-[p]_r|^{\frac{3+2\alpha}{3}}\Big)^{\frac{9-6\alpha}{3+2\alpha}}
\end{equation}
for any $r>0$. By H\"older's inequality and the Sobolev-Poincar\'e inequality, we observe that
\begin{align*}
       \mathcal{D}(r)&= \frac{1}{r^{6-4\alpha}}\int_{Q_r}|p-[p]_r|^{\frac{3}{2}}\,dx\,dt \\
       &\leq  \frac{1}{r^{6-4\alpha}} \int_{-r^{2\alpha}}^{0} \Big(\int_{B_r}|p-[p]_r|^{\frac{3+2\alpha}{4}}\,dx\Big)^{\frac{8\alpha-6}{3+2\alpha}} \Big(\int_{B_r}|p-[p]_r|^{\frac{3+2\alpha}{3}}\,dx\Big)^{\frac{9-6\alpha}{3+2\alpha}}\,dt \\
       &\leq  \frac{C}{r^{6-4\alpha}} \int_{-r^{2\alpha}}^{0} \Big(r^{\frac{3+2\alpha}{4}}\int_{B_r}|\nabla p|^{\frac{3+2\alpha}{4}}\,dx\Big)^{\frac{8\alpha-6}{3+2\alpha}} \Big(\int_{B_r}|p-[p]_r|^{\frac{3+2\alpha}{3}}\,dx\Big)^{\frac{9-6\alpha}{3+2\alpha}}\,dt\\
       &\leq C r^{6\alpha-\frac{15}{2}} \Big(\int_{Q_r} |\nabla p|^{\frac{3+2\alpha}{4}}\Big)^{\frac{8\alpha-6}{3+2\alpha}} \Big(\int_{Q_r}|p-[p]_r|^{\frac{3+2\alpha}{3}}\Big)^{\frac{9-6\alpha}{3+2\alpha}}
\end{align*}
as claimed. A direct consequence of the above claim \eqref{pressure} is the following estimate on $II$
\begin{equation}\label{step1}
\begin{aligned}
    II =\theta^{4\alpha-\frac{3}{2}N} \mathcal{D}(r_0) &\leq C \theta^{(4\alpha-\frac{3}{2})N}r_0^{6\alpha-\frac{15}{2}} \Big(\int_{Q_{r_0}} |\nabla p|^{\frac{3+2\alpha}{4}}\Big)^{\frac{8\alpha-6}{3+2\alpha}} \Big(\int_{Q_{r_0}}|p-[p]_r|^{\frac{3+2\alpha}{3}}\Big)^{\frac{9-6\alpha}{3+2\alpha}}\\
    &\leq C \rho^{(4\alpha-\frac{3}{2})N\zeta+(6\alpha-\frac{15}{2})\eta+(L-\gamma)}\varepsilon
    \end{aligned}
\end{equation}
thanks to the assumption \eqref{MainAssumption}. By our future choice of $\eta$ in \eqref{eta2}, we will further obtain
\begin{equation}
    II\leq C\rho^{J_2}\varepsilon
\end{equation}
with
\begin{equation*}
    J_2:=\left(4\alpha-\frac{3}{2}\right)N\zeta+\left(6\alpha-\frac{15}{2}\right)\eta + L-\gamma.
\end{equation*}
\noindent\textbf{5. Control over $I$. }
Now we have the estimates for $\mathcal{A}(\rho)$ and $\mathcal{E}(\rho)$, so we are in the position to complete the estimate on $\mathcal{C}(r_k)$. By Lemma \ref{interpolation}, we start by
\begin{equation}
    \begin{split}
        \mathcal{C}(r_k)\leq C \Big(\frac{\rho}{r_k}\Big)^{\frac{15}{2}-6\alpha}\mathcal{A}^{1/2}(\rho)\mathcal{E}(\rho)+C \Big(\frac{r_k}{\rho}\Big)^{6\alpha-3}\mathcal{A}^{3/2}(\rho).
    \end{split}
\end{equation}
Noting that $r_k=\rho^{\eta+k\zeta},$ thanks to the estimates for $\mathcal{A}(\rho)$ and $\mathcal{E}(\rho)$ we get
\begin{equation*}
     \mathcal{C}(r_k)\leq C \varepsilon^{\frac{9}{6+4\alpha}}\bigg(\rho^{(\eta +k\zeta)(6\alpha-\frac{15}{2}) +(L-\gamma)\frac{9+4\alpha}{6+4\alpha}}+\rho^{(\eta+k\zeta)(6\alpha-3)+(L-\gamma)\frac{9}{6+4\alpha}-\frac{9}{2}}\bigg).
\end{equation*}
Then some elementary comparison between exponents of $\rho$ leads to
\begin{equation}\label{eq:II}
\begin{split}
     I &=\sum_{i=1}^{N}\theta^{(4\alpha-\frac{3}{2})(i-1)+4\alpha-6}\mathcal{C}(r_{N-i}) \\
     &\leq C\varepsilon^{\frac{9}{6+4\alpha}}\rho^{-\frac{9}{2}\zeta+(\eta+N\zeta)(6\alpha-\frac{15}{2})+(L-\gamma)\frac{9}{6+4\alpha}} \sum_{i=1}^{N}
    \bigg(\rho^{(6-2\alpha) i\zeta+(L-\gamma)\frac{2\alpha}{3+2\alpha}}       +\rho^{i\zeta(\frac{3}{2}-2\alpha)+\frac{9}{2}(\eta+N\zeta)-\frac{9}{2}} \bigg) \\
    &\leq C \varepsilon^{\frac{9}{6+4\alpha}}\rho^{-\frac{9}{2}\zeta+(\eta+N\zeta)(6\alpha-\frac{15}{2})+(L-\gamma)\frac{9}{6+4\alpha}} \bigg(\rho^{(6-2\alpha)\zeta+(L-\gamma)\frac{2\alpha}{3+2\alpha}}+\rho^{N\zeta(\frac{3}{2}-2\alpha)+\frac{9}{2}(\eta+N\zeta)-\frac{9}{2}}\bigg). \\
\end{split}
\end{equation}
In particular, if we further determine $\eta\geq 1$ via the relation
\begin{equation}\label{eta1}
    (6-2\alpha)\zeta+(L-\gamma)\frac{2\alpha}{3+2\alpha} = N\zeta(\frac{3}{2}-2\alpha)+\frac{9}{2}(\eta+N\zeta)-\frac{9}{2},
\end{equation}
or equivalently,
\begin{equation}\label{eta2}
    \eta = \frac{2}{9}\left((6-2\alpha)\zeta+(2\alpha-6)N\zeta+\frac{9}{2}+(L-\gamma)\frac{2\alpha}{3+2\alpha}\right),
\end{equation}
which lets both exponents of $\rho$ in \eqref{eq:II} coincide with each other,
then we further obtain
\begin{equation}
    I\leq C\varepsilon^{\frac{9}{6+4\alpha}}\rho^{J_3}
\end{equation}
where $J_3$ is defined by
\begin{equation*}
    J_3:=\frac{-16\alpha^2+56\alpha-51}{6}\zeta +\frac{(4\alpha-5)(4\alpha-3)}{6}N\zeta +\frac{16\alpha^2-8\alpha+27}{6(3+2\alpha)}(L-\gamma)+\frac{3(4\alpha-5)}{2}.
\end{equation*}

\textbf{Step 6. Choosing parameters.} In short, we have obtained
\begin{equation}
    \begin{split}
        \mathcal{C}(r_N)+\mathcal{D}(r_N)+\mathcal{T}(r_N) &\leq C\rho^{J_1}\varepsilon^{\frac{3}{3+2\alpha}} + C\rho^{J_2}\varepsilon+C\rho^{J_3}\varepsilon^{\frac{9}{6+4\alpha}}
    \end{split}
\end{equation}
If we can choose $(\eta,\zeta,N,\gamma)$ such that $J_i\geq 0$ for all $i\in\{1,2,3\},$ then we obtain
\begin{equation}\label{conclusion}
    \mathcal{C}(r_N)+\mathcal{D}(r_N)+\mathcal{T}(r_N) \leq C \varepsilon^{\frac{3}{3+2\alpha}} \leq \varepsilon_0
\end{equation}
by taking sufficiently small $\varepsilon>0,$ because $\rho$ is strictly smaller than $1.$ By Proposition~\ref{prop_epsilon}, we conclude that $z$ is a regular point. Therefore it suffices to find such parameters $(\eta,\zeta,N,\gamma).$

Since we have also used the assumptions that $\eta\geq 1$ and \eqref{gamma2} for the estimates in the previous steps, actually the quadruple $(\eta,\zeta,N,\gamma)$ should satisfy the below five conditions corresponding to $J_1 \geq 0$, $J_2 \geq 0 $, $J_3 \geq 0$, $\eta \geq 1$, and \eqref{gamma2}:
\begin{equation}\label{five}
    \begin{split}
    \gamma &\leq L + \frac{9(3+2\alp)(4\alp-2)}{16\alp^2-8\alp+27}+\frac{(3+2\alp)(4\alpha-2)(4\alpha-3)}{(16\alp^2-8\alp+27)}N\zeta + \frac{2(3+2\alp)(4\alp-2)(6-2\alp)}{16\alp^2-8\alp+27}\zeta, \\
    \gamma &\leq L + \frac{9(4\alpha-5)(3+2\alpha)}{16\alpha^2-8\alpha+18}+\frac{(3+2\alpha)(16\alpha^2-44\alpha+51)}{16\alpha^2-8\alpha+18}N\zeta+\frac{2(3+2\alpha)(4\alpha-5)(6-2\alpha)}{3(16\alpha^2-8\alpha+18)}\zeta,  \\
    \gamma &\leq L + \frac{9(4\alpha-5)(3+2\alpha)}{16\alpha^2-8\alpha+27}+\frac{(3+2\alpha)(4\alpha-5)(4\alpha-3)}{16\alpha^2-8\alpha+27}N\zeta+\frac{9(3+2\alpha)(4\alpha-5)}{16\alpha^2-8\alpha+27}\zeta,\\
    \gamma &\leq L +\frac{(3+2\alp)(\alp-3)}{\alpha}N\zeta + \frac{(3+2\alpha)(3-\alp)}{\alpha}\zeta, \\
    \gamma &\leq \frac{4\alpha-3}{4\alpha}L.
    \end{split}
\end{equation} Note that we used the choice of $\eta$ in \eqref{eta2} and also that the conditions are stated in terms of the upper bound of $\gamma$ because we hope to maximize $\gamma$ to improve the given bound $L=L(\alpha).$ Then one can check that it is appropriate to set
\begin{equation}\label{N}
    N\zeta=\frac{27(4\alpha-5)}{64\alpha^3-272\alpha^2+300\alpha-369}.
\end{equation}
The above choice of $N\zeta$ can be derived from the following heuristic argument. The other three upper bounds for $\gamma$ in \eqref{five} are expected to be relatively big compared with the two conditions related to $J_2\geq 0$ and $J_3\geq 0$. Thus it suffices to consider those two because they are the only plausible candidates for the most restrictive one. We notice that $\zeta$ vanishes as $N$ grows, so we temporarily assume $\zeta\sim 0$ while $N\zeta$ is still significant. Equating the second upper bound and the third one for $\gamma$, we get \eqref{N} as claimed.

Thanks to the particular choice of $N\zeta$ we can finish the selection process rigorously as follows. For any fixed $\gamma\geq0$ which satisfies
\begin{equation}
\begin{split}
    \gamma &< L(\alpha)+\frac{9(4\alpha-5)(3+2\alpha)}{16\alpha^2-8\alpha+27}+\frac{(3+2\alpha)(4\alpha-5)(4\alpha-3)}{16\alpha^2-8\alpha+27}N\zeta \\
    &= L(\alpha)-\left( \frac{27(3+2\alpha)(4\alpha-5)^2(3-4\alpha)}{(16\alpha^2-8\alpha+27)(64\alpha^3-272\alpha^2+300\alpha-369)}+\frac{9(3+2\alpha)(5-4\alpha)}{16\alpha^2-8\alpha+27}\right) \\
    &=L(\alpha)-J(\alpha),
\end{split}
\end{equation}
we choose sufficiently large $N>0$ such that
\begin{equation*}
    \zeta = \frac{1}{N}\cdot \frac{27(4\alpha-5)}{64\alpha^3-272\alpha^2+300\alpha-369} \leq \frac{16\alpha^2-8\alpha+27}{9(3+2\alpha)(5-4\alpha)} \Big(L(\alpha)-J(\alpha)-\gamma\Big).
\end{equation*}
This leads to \eqref{five} as targeted. In other words, such a specific choice of $N$ allows $\zeta$ to become as negligible as we need in the previous 
 argument for $N\zeta$. This rigorously justifies that our fully determined quadruple $(\eta,\zeta,N,\gamma)$ is admissible. Finally, we select sufficiently small $\rho<1$ that the common ratio $\theta=\rho^\zeta$ is smaller than $1/2$, which is crucial for our auxiliary lemmas (Lemma \ref{interpolation}, and Lemma \ref{lem:pressure}) to be well-applied.  We eventually conclude that
\begin{equation*}
    \mathcal{C}(r_N)+\mathcal{D}(r_N)+\mathcal{T}(r_N)\leq C\varepsilon^{\frac{3}{3+2\alpha}}\leq \varepsilon_0
\end{equation*}
with sufficiently small $\varepsilon>0,$ as desired. Since $\gamma < L(\alpha)-J(\alpha)$ is arbitrary, the proof is complete. \end{proof}

\appendix

\section{Appendix}

\begin{proposition}
Suppose that $(u,p)$ is a suitable weak solution to \eqref{fNSE} in $\R^3\times(0,T)$. Then we have
\begin{equation}\label{pressure5/4}
    \int_{\R^3\times(0,T)}|\nabla p|^{\frac{3+2\alpha}{4}}\,dx\,dt < \infty.
\end{equation}
\end{proposition}
\begin{proof}
Denote by $R_k$ the $k$-th Riesz transform with symbol $i\xi_k/|\xi|$. Since $\partial_i$ and $\sum R_i R_j$ commute (consider their Fourier symbols), we write
\begin{equation*}
    \partial_k p= \sum R_i R_j (\partial_k u_i u_j + u_i \partial_k u_j)
\end{equation*}
using the divergence-free condition of $u.$ Since $\sum R_i R_j$ is $L^p$-bounded for any $1<p<\infty,$ we obtain
\begin{equation}\label{riesz}
    \|\nabla p\|_{L^p(\R^3)} \leq C_{p} \||u||\nabla u|\|_{L^p(\R^3)}.
\end{equation}
Observing $\|\nabla u\|_{L^{\frac{6}{5-2\alpha}}(\R^3)} \leq C \|(-\Delta)^{\frac{\alpha}{2}}u\|_{L^2(\R^3)}$ and $\|u\|_{L^{\frac{6}{3-2\alpha}}(\R^3)}\leq C \|(-\Delta)^{\frac{\alpha}{2}}u\|_{L^2(\R^3)}$ by fractional Sobolev's embeddings, we use \eqref{riesz}, H\"older's inequality, and Lebesgue interpolation to compute
\begin{equation*}
    \begin{split}
        \int_{0}^{T}\int_{\R^3}|\nabla p|^{r} \,dx\,dt &\leq C \int _{0}^{T} \int_{\R^3} |(u\cdot\nabla) u|^{r} \,dx\,dt  \\
        &\leq C \int_{0}^{T} 
         \Big(\int_{\R^3}|u|^{\frac{6}{6-r(5-2\alpha)}}\Big)^{\frac{6-r(5-2\alpha)}{6}} \Big(\int_{\R^3} |\nabla u|^{\frac{6}{5-2\alpha}}\Big)^{\frac{r(5-2\alpha)}{6}}\,dt \\
        &\leq C \int_{0}^{T} \Big( \|u\|_{L^2}^{\theta}\|u\|_{L^{\frac{6}{3-2\alpha}}}^{1-\theta}\Big)^{r} \Big(\int |(-\Delta)^{\frac{\alpha}{2}}u|^2\Big)^{\frac{r}{2}} \,dt \\
    \end{split}
\end{equation*}
for $r>0$ and $0<\theta<1$ such that $2<\frac{6r}{6-r(5-2\alpha)}<\frac{6}{3-2\alpha}.$ By the uniform in time $L^2$-boundedness of $u$, fractional Sobolev's embedding, and H\"older's inequaltiy again, we get
\begin{equation*}
    \begin{split}
        \int_{0}^{T}\int_{\R^3}|\nabla p|^{r} \,dx\,dt
        &\leq C \Big(\int_{0}^{T} \|u\|_{L^{\frac{6}{3-2\alpha}}}^{2} \,dt\Big)^{\frac{r(1-\theta)}{2}} \Big(\int_{0}^{T}\Big(\int|(-\Delta)^{\frac{\alpha}{2}}u|^2\Big)^{\frac{r}{2-r(1-\theta)}}\,dt\Big)^{\frac{2-r(1-\theta)}{2}} \\
        &\leq C \Big(\int_{\R^3\times(0,T)} |(-\Delta)^{\frac{\alpha}{2}}u|^2 \,dx\,dt \Big)^{\frac{r(1-\theta)}{2}} \Big(\int_{0}^{T}\Big(\int|(-\Delta)^{\frac{\alpha}{2}}u|^2\Big)^{\frac{r}{2-r(1-\theta)}}\,dt\Big)^{\frac{2-r(1-\theta)}{2}} \\
        &\leq C \Big(\int_{0}^{T}\Big(\int|(-\Delta)^{\frac{\alpha}{2}}u|^2\Big)^{\frac{r}{2-r(1-\theta)}}\,dt\Big)^{\frac{2-r(1-\theta)}{2}}
    \end{split}
\end{equation*}
for $r>0$ and $0<\theta<1$ which further satisfies $\frac{2}{r(1-\theta)}>1$ to apply H\"older's inequality for the first inequality in the above calculation. Now we set
\begin{equation*}
    \begin{split}
        r:=\frac{3+2\alpha}{4}, \quad 
        \theta:= \frac{4\alpha-2}{3+2\alpha}.
    \end{split}
\end{equation*}
For such $r$ and $\theta$, not only all the previous computations are valid but also we have $\frac{r}{2-r(1-\theta)}=1$ so that we conclude that
\begin{equation*}
    \int_{\R^3\times (0,T)} |\nabla p|^{\frac{3+2\alpha}{4}} \,dx\,dt \leq C \Big(\int_{\R^3\times(0,T)} |(-\Delta)^{\frac{\alpha}{2}}u|^2 \,dx\,dt \Big)^{\frac{3+2\alpha}{8}} < \infty,
\end{equation*}
as desired. The proof is complete.
\end{proof}

\end{document}